\documentclass[preprint,12pt]{elsarticle}

\usepackage{amssymb}

\usepackage{amsfonts}
\usepackage{amsmath}
\usepackage{amsthm}
\usepackage{amssymb}
\usepackage{graphicx}
\usepackage{bbm}
\usepackage{tikz}
\usetikzlibrary{decorations.pathreplacing,calligraphy, shapes.geometric}

\usetikzlibrary{chains,arrows.meta,quotes,bending}

\usepackage{enumerate}
\usepackage{enumitem}
\usepackage{latexsym}
\usepackage{xcolor}
\usepackage{caption} 

\usepackage[colorinlistoftodos,prependcaption,textsize=tiny]{todonotes}
\setuptodonotes{fancyline, color=green!20}

\usepackage{hyperref}

\usepackage{tikz}
\usetikzlibrary{decorations.pathreplacing,calligraphy}
\theoremstyle{thmstyleone}%
\newtheorem{theorem}{Theorem}
\theoremstyle{thmstyletwo}%
\newtheorem{example}{Example}%
\newtheorem{remark}{Remark}%

\theoremstyle{thmstylethree}%
\begin{document}

\begin{frontmatter}



\title{An infinitesimal generator approach on weak convergence of regulated multi-class matching systems}


\author[inst1]{Bowen Xie\corref{cor1}}

\affiliation[inst1]{organization={Department of Mathematics, 
College of Engineering and Polymer Science, 
The University of Akron},
            city={Akron},
            state={OH},
            postcode={44325-4002}, 
            country={US}
            }

\cortext[cor1]{bxie@uakron.edu}

\author[inst2]{José E. Figueroa-López\corref{cor2}}

\affiliation[inst2]{organization={Department of Statistics and Data Science, Washington University in St. Louis},
             city={St. Louis}, 
             postcode={63130}, 
             state={MO},
             country={US}}
             
\cortext[cor2]{figueroa-lopez@wustl.edu}

\begin{abstract}

We consider a regulated multi-class instantaneous matching system with reneging, in which each event requires $K \geq 2$ distinct impatient agents who wait in their respective queues. Each agent class is subject to a buffer capacity, allowing for the special case without buffers. 
Due to the instantaneous matching behavior, at any give time, at least one category has an empty queue. 
Under the Markovian assumption, the system dynamics are described by a Markov chain with innovative rate matrices that capture all possible queue configurations across all classes. 
To effectively circumvent the structural challenges introduced by instantaneous matching, we establish a non-trivial yet tractable diffusion approximation under heavy traffic conditions by leveraging the infinitesimal generator in conjunction with appropriate regulation and boundary conditions. 
This asymptotic analysis offers a direct explanation of the dynamics of the regulated coupled heavy-traffic limiting process. 
Furthermore, we demonstrate the connection between the diffusion-scaled limit derived from the generator approach and the one established in the literature. The latter is typically described by a regulated coupled stochastic integral equation.

\end{abstract}




\end{frontmatter}



\section{Introduction}
\label{sec: Introduction}

Matching etiquette has drawn significant attention in the context of high-efficiency platforms that facilitate multi-party matching. Such platforms arise in various domains, including the production of technological and pharmaceutical products, volunteer coordination in non-profit organizations (cf. \cite{zychlinski2024managing}, \cite{doval2025efficiency}), jury selection in judicial systems, and market making in high-frequency trading (cf. \cite{guo2015dynamics}, \cite{guo2017optimal}). 
These instances share a common structural feature: the need to coordinate and match multiple classes of participants in a timely and efficient manner. This phenomenon can thus be abstracted into a general setting where an event requires the simultaneous participation of multiple types of agents. 
Agents from each distinct class arrive sequentially over time and join their respective queues. To successfully initiate an event, one agent from each class is required, and once matched, the involved agents exit the system immediately. The matching operates under a first-come-first-matched (FCFM) discipline. These agents may exhibit impatience, abandoning the system if they are not matched within their individual patience thresholds. Given the instantaneous nature of matching, it is evident that all queues cannot simultaneously contain a positive number of agents at any given moment, namely, at least one queue must be empty. Such queueing systems are referred to as multi-component matching queues with impatient components (cf. 
\cite{plambeck2006optimal}, \cite{buke2015stabilizing}, \cite{gurvich2015dynamic}, \cite{mairesse2016stability}, \cite{nazari2019reward}, 
\cite{jonckheere2023generalized}, \cite{xie2024multi}).

In the study of queueing theory, various approaches have been developed to analyze diffusion approximations of queueing systems. One particularly compelling method involves the use of infinitesimal generators, especially in the context of Markovian models (cf. \cite{harrison1973heavy}, \cite{harrison1978diffusion}, \cite{martins1996heavy}, \cite{kumar2007integrating}, \cite{pang2007martingale}, \cite{budhiraja2011multiscale}, \cite{xie2024diffusion}). This approach circumvents the need to explicitly construct queue length processes or prove tightness, thereby significantly simplifying the analysis, especially for complex systems, by leveraging the structural properties of the underlying Markov processes.

In this article, we introduce a multi-component matching system with perishable components and restricted buffer capacities. 
To derive its heavy-traffic diffusion approximation, we conduct a direct analysis based on the infinitesimal generator method, highlighting the model’s non-trivial structural intricacies. 
Our main contributions are threefold:
First, the rate matrices are challenges so that they need to precisely capture the dynamics of states of each queue. 
For instance, an arbitrary new arrival may result in one of three outcomes: an instantaneous match if compatible agents from other categories are available, admission to the waiting room if capacity permits, or rejection due to a full queue, which are mainly characterized by a sequence of indicator functions associated with the status of queues. 
Second, the infinitesimal generators of both the discrete-state process and its diffusion limit are non-trivial, particularly in the presence of buffer capacities $b_i\in(0, \infty]$ for each queue $i\in\{1, \cdots, K\}$. 
We assume that the queue lengths are constrained within the interval $[0, b_i]$, where $b_i = \infty$ represents an infinite buffer, namely, no buffer constraint is imposed. 
Naturally, this setting necessitates careful specification of regulation mechanisms and boundary conditions at the buffer limits (cf. \cite{xie2022topics}). Moreover, the generators must appropriately capture mixed boundary behaviors, such as when $s_i < b_i$ for some categories $i$'s and $s_j = b_j$ for other $j$'s. 
For brevity, we provide a more detailed discussion in Section \ref{sec: Asymptotic analysis}. Utilizing the generator approach, one can leverage established results (Theorem 6.1 of Chapter 1 and Theorem 2.11 of Chapter 4 in \cite{ethier2009markov}) within an appropriate function space. 
Our findings further highlight the critical role of regulation and boundary conditions in the diffusion approximation framework.
Third, the connections between the diffusion limiting process associated with the infinitesimal generator and the heavy-traffic limit obtained from \cite{xie2024multi}, which is derived from an integral representation and its corresponding functional central limit theorem (FCLT), are carefully discussed. 
We also exhibit a simple example of double-ended queue when the number of categories $K = 2$ to illustrate the equivalence between these two formulations of the limiting process.

The stochastic matching queue analysis has gained a lot of attention in recent literature \cite{mairesse2020editorial}. 
The double-ended system has been well studied recently, and corresponding control problems have been concerned (cf.
\cite{conolly2002double},  
\cite{liu2015diffusion}, 
\cite{liu2019diffusion},   \cite{liu2021admission},  \cite{lee2021optimal}, etc.). 
In \cite{conolly2002double}, the effect of reneging is studied in the context of double-ended queues, where each demands service from the other to provide a theoretical but brief numerical assessment of operational consequences. 
In \cite{liu2015diffusion}, under a suitable asymptotic regime, they established fluid and diffusion approximations for the queue length process, which are characterized by an ordinary differential equation and time-inhomogeneous asymmetric Ornstein-Uhlenbeck process. They also exhibited the interchangeability of the heavy traffic and steady state limits. 
In \cite{liu2019diffusion}, they studied a double-ended system with two classes of impatient customers and established simple linear asymptotic relationships between the diffusion-scaled queue length process and the diffusion-scaled offered waiting time processes in heavy traffic. 
They also showed that the diffusion-scaled queue length process converges weakly to a diffusion process that admits a unique stationary distribution.
In \cite{lee2021optimal}, they studied a double-ended system having backorders and customer abandonment and determined the optimal (nonstationary) production rate over a finite time horizon to minimize the costs incurred by the system under some cost structure.

The generalized multi-class matching has also been studied in the literature with different formulations through various aspects (cf. \cite{harrison1973assembly}, \cite{plambeck2006optimal}, \cite{gurvich2015dynamic}, etc.). 
In \cite{harrison1973assembly}, Harrison studied a model with an assembly-like behavior to produce a product with several components and developed limit theorems for the appropriately normalized versions of the associated vector waiting time process in heavy traffic. 
The model contains $K\geq2$ independent renewal input processes. 
The server requires one input component of each category $j = 1, \cdots, K$, and once the server has all the required components, it takes a random processing time to finish the product.
\cite{plambeck2006optimal} introduced a model with order queues and component queues and studied a control problem of an assembly-to-order system with a high volume of prospective customers arriving per unit time, where multiple different components are instantaneously assembled into different finished products, and the control problem is developed so that they can maximize the expected infinite-horizon discounted profit by choosing product prices, component production capacities, and a dynamic policy for sequencing customer orders for assembly. 
In \cite{gurvich2015dynamic}, the authors studied a matching system with instantaneous processing, where a system manager could control which matches to execute given multiple options and addressed the problem of minimizing finite-horizon cumulative holding costs. 
They established a multi-dimensional imbalance process to characterize the matching model and devised a myopic discrete-review matching control, which is shown to be asymptotically optimal in heavy traffic. 
More results on matching systems can be found in \cite{green1985queueing}, \cite{adan2009exact}, \cite{adan2018reversibility}, and \cite{fazel2018approximating}. 
Applications to ride-sharing systems are explored in \cite{kashyap1966double} and \cite{ozkan2020dynamic}, while organ transplant systems are studied in \cite{boxma2011new} and \cite{khademi2021asymptotically}, and blood bank operations are examined in \cite{bar2017blood}. Related models in production-inventory systems are discussed in \cite{kaspi1983inventory}, \cite{perry1999perishable}, \cite{xie2024long}, and \cite{lee2021optimal}.

The remainder of the paper is organized as follows.
In Section \ref{sec: basic model}, we introduce the matching system with perishable components and restricted buffer capacities, along with the basic modeling assumptions.
Section \ref{sec: asymptotic framework} presents the asymptotic framework and derives the corresponding infinitesimal generator. 
In Section \ref{sec: Asymptotic analysis}, we employ the generator approach, together with appropriate regulation and boundary conditions, to develop a diffusion approximation under heavy traffic. We also discuss the connection between the limiting process derived via direct analysis and that obtained through the functional central limit theorem (FCLT).

\textbf{Notation.} 
Let $\mathbb{N}$ represent the set of positive integers. 
Let $\mathbb{R}$ denote the one-dimensional Euclidean space. For $0< T\leq \infty$, let $D[0, T]$ denote the Skorokhod space of functions with right continuous and left limits (RCLL). 
The uniform norm on $[0, T]$ for a stochastic process $X$ in $D[0, T]$ is defined by $\|X\|_T = \sup_{t\in[0, T]}|X(t)|$. 
Let $B(E)$ denote the Banach space of all bounded measurable functions on $E$, and the uniform norm is given by $\|f\|_{B(E)} = \sup_{x\in E}|f(x)|$ for $f\in B(E)$. 
For any operator $A$ defined in a Banach space, we let $\mathrm{D}(A)$ denote its domain. 
For $n\in\mathbb{N}$, let $L_n$ and $L$ be Banach spaces (with norm denoted by $\|\cdot\|$) and $\pi_{n}:L\to L_n$ denote a bounded linear transformation. We write a notion of convergence as
\begin{equation}
    f_n\to f, 
    \label{eq: notion of convergence}
\end{equation}
if $f_n\in L_n$ for each $n\geq 1$, $f\in L$, and $\lim_{ n\to\infty} \|f_n - \pi_n f\| = 0$. 
Throughout, we use $\Rightarrow$ to denote weak convergence in $D[0, T]$. 
For any real number $a$, $a^+ = \max\{a, 0\}$ and $a^- = \max\{-a, 0\}$. For any two real numbers $a$ and $b$, $a\wedge b = \min\{a, b\}$ and $a\vee b = \max\{a, b\}$.

\section{Basic Model}
\label{sec: basic model}

We consider the multi-component matching queue with perishable components as introduced in \cite{xie2024multi}. In this setting, we further impose some fixed bandwidth capacity $b_i\in(0, \infty]$ for each category $i\in\{1, \cdots, K\}$ and $K\geq 2$. Upon arrival, a component checks the capacity of its corresponding queue, and if the queue is full, the component is discarded  from the system. 
If $b_i = \infty$, the capacity of the $i$th queue admits no restriction. 
Figure \ref{fig: Schematic diagram of matching operation} exhibits the sample system status at some time instant, where the third queue for green triangle items is empty and the red solid, blue dotted, and green dashed lines are the buffer capacities for respective queues. 
If a red square component arrives in the system and the red queue is at full capacity, the component will be discarded as in the schematic diagram.

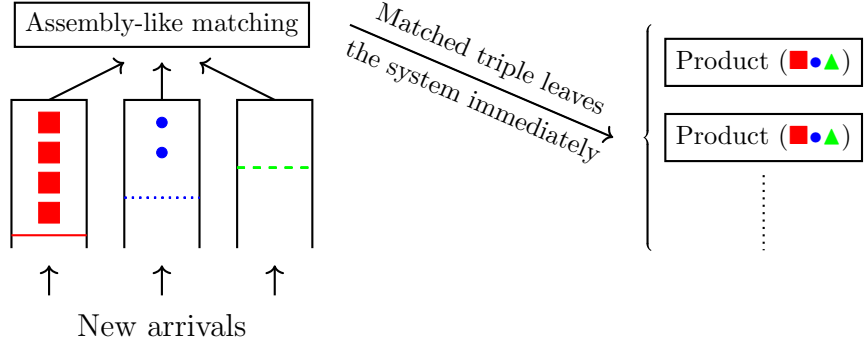
\begin{figure}[h!]
    \centering
    \begin{tikzpicture}[circ/.style={shape=circle, fill, inner sep=2pt, draw, node contents=}]
        \node at (0,2) [rectangle,draw, thick] (product) {\footnotesize Assembly-like matching};
        
        \draw[->, thick] (-1.5, 1) -- (-0.5, 1.5);
        \draw[thick] (-2, 1) rectangle (-1, -1);
        \node at (-1.5, 0.7) [red, rectangle, draw, fill]{}; 
        \node at (-1.5, 0.3) [red, rectangle, draw, fill]{};
        \node at (-1.5, -0.1) [red, rectangle, draw, fill]{};
        \node at (-1.5, -0.5) [red, rectangle, draw, fill]{};
        \draw[thick, red] (-2, -0.8) rectangle (-1, -0.8); 
        
        \draw[->, thick] (0, 1) -- (0, 1.5);
        \draw[thick] (-0.5, 1) rectangle (0.5, -1);
        \node at (0, 0.7) [blue, circ, draw, fill, scale=0.7];
        \node at (0, 0.3) [blue, circ, draw, fill, scale=0.7];
        \draw[thick, blue, dotted] (-0.5, -0.3) rectangle (0.5, -0.3); 

        \draw[->, thick] (1.5, 1) -- (0.5, 1.5);
        \draw[thick] (1, 1) rectangle (2, -1);
        \draw[thick, dashed, green] (1, 0.1) rectangle (2, 0.1);

        \draw [->, thick] (2.5, 2) -- (6, 0.5);
        \draw node (nodeA) at (2.5, 2) {};
        \draw node (nodeB) at (6, 0.5) {};
        \draw (nodeA) -- (nodeB) node [midway, above, sloped] (TextNode) {\footnotesize Matched triple leaves};
        \draw (nodeA) -- (nodeB) node [midway, below, sloped] (TextNode) {\footnotesize the system immediately};
        
        \draw [decorate, decoration = {calligraphic brace}, thick] (6.5,-1) --  (6.5,2);
        \node at (8,1.5) [rectangle,draw, thick] (product) {\footnotesize Product ($\textcolor{red}{\blacksquare} \textcolor{blue}{\bullet} \textcolor{green}{\blacktriangle}$)};
        \node at (8,0.5) [rectangle,draw, thick] (product) {\footnotesize Product ($\textcolor{red}{\blacksquare} \textcolor{blue}{\bullet} \textcolor{green}{\blacktriangle}$)};
        \draw [dotted, thick] (8, 0) -- (8, -1);

        \draw[white, ultra thick]  (-2.2, -1) -- (2.2, -1);
        
        \draw[->, thick] (-1.5, -1.6) -- (-1.5, -1.2);
        \draw[->, thick] (0, -1.6) -- (0, -1.2);
        \draw[->, thick] (1.5, -1.6) -- (1.5, -1.2);
        \draw node at (0, -2) {New arrivals};
    \end{tikzpicture}
    \caption{Schematic diagram of a regulated matching operation for a product made of components from three distinct categories $\textcolor{red}{\blacksquare}$, $\textcolor{blue}{\bullet}$, $\textcolor{green}{\blacktriangle}$. }
    \label{fig: Schematic diagram of matching operation}
\end{figure}

To establish such a model, let $Q_i(t)$ denote the queue length of the $i$th component at time $t>0$, $A_i(t)$ represent the number of arrivals of category $i$ by time $t$, and $G_i(t)$ capture the number of abandoned component of category $i$ by time $t$. It is straightforward to establish the following queue length process for the $i$th component: 
\begin{equation}\label{eq: basic queue length}
    Q_i(t) = Q_i(0) + \int_0^t \mathbbm{1}_{[Q_i(s) < b_i]}dA_i(s) - G_i(t) - R(t), 
\end{equation}
where $Q_i(0)$ denotes the initial components of category $i$, $R(t)$ represent the total number of matches by time $t$ such that
\begin{equation}\label{eq: basic matching completion}
    R(t) = \min_{j\in[0, K]}
    \left\{
    Q_j(0) + \int_0^t \mathbbm{1}_{[Q_j(s) < b_j]}dA_j(s) - G_j(t)
    \right\}. 
\end{equation}
Note that the integration term counts the number of unblocked components by time $t$, which can be decomposed into two terms: the number of arrivals and the number of blocked components. 
We further define the number of blocked components by time $t$, $L_i(t)$, such that
\[
L_i(t) = \int_0^t \mathbbm{1}_{[Q_i(s) \geq b_i]} dA_i(s). 
\]
Therefore, \eqref{eq: basic queue length} can be rewritten as
\begin{equation}\label{eq: basic queue length with decomposed arrivals}
    Q_i(t) = Q_i(0) + A_i(t) - L_i(t) - G_i(t) - R(t).  
\end{equation}
Our objective is to study the heavy traffic behavior of the system when all the components arrive frequently. It is natural to perform the direct analysis of such a system under Markovian assumptions.

\section{Asymptotic Framework}
\label{sec: asymptotic framework}

To carry out the asymptotic analysis, we consider a sequence of independent systems parameterized by $n > 0$ such that the arrival rates get large without bounds as $n\to\infty$. 
Let the queue length process $Q^n(\cdot) = (Q_1^n(\cdot), \cdots, Q_K^n(\cdot))^\intercal$ be the state process. For each component $i\in \{1, \cdots, K\}$, let $A_i^n(\cdot)$ and $G_i^n(\cdot)$ be two independent processes representing the number of arrivals and abandonments of category $i$ in the $n$th system. 
We assume that $A_i^n(\cdot)$ follows a Poisson process in $D([0, \infty), \mathbb{R})$ with arrival rate $\lambda_i^n > 0$, and $\{A_k^n\}_{1\leq k\leq K}$ are all independent with each other. 
Moreover, we assume $\lambda_{i}^{n}\to\infty$ as $n\to\infty$ for each $i$.
We also assume that the abandonment processes follow independent Poisson processes with respective parameter $\delta_i^n > 0$ such that it is constructed by
\begin{equation}
    G_{i}^{n}(t) := N_i\left(\delta_{i}^{n}\int_0^t Q_{i}^{n}(s)ds\right), 
    \label{abandonment process}
\end{equation}
where $\delta_{i}^{n}>0$ is a constant and $N_i$'s are independent unit rate Poisson processes. 
We assume $\lim_{n\to\infty}\delta_i^n = \delta_i$, where $\delta_i>0$ is a real number. 
More precisely, one can think of the patience time of a component as independent of its arrival time as well as the arrival times and patience times of those components who arrived earlier, and they are also independent of everything else in the system. 
Additionally, let $L_i^n(\cdot) = \int_0^\cdot \mathbbm{1}_{[Q_i^n(s) \geq b_i^n]}dA_i^n(s)$ denote the number of blocked components of category $i$ in the $n$th system. 

We describe other basic assumptions and exhibit the heavy traffic assumption as follows. 

\textbf{Assumption 1} (Buffer conditions). 
For each $i\in\{1, \cdots, K\}$, we assume $\lim_{n\to\infty} b_i^n / \sqrt{n} = b_i\in(0, \infty]$. 

\textbf{Assumption 2} (Heavy-traffic condition). 
For each $i\in \{1, \cdots, K\}$, there exists a constant $\lambda_0 > 0$ such that 
\begin{equation}
    \lim_{n\to\infty} \frac{\lambda_i^n - \lambda_0 n}{\sqrt{n}} = \beta_i, 
\end{equation}
where $\beta_i$ is a real number. 

To address the system behavior in heavy traffic, we further introduce the following centered and scaled quantities: 
\begin{equation}
\begin{aligned}
    \hat{Q}_i^n(t) &:= \frac{Q_i^n(t)}{\sqrt{n}}, \quad \hat{A}_i^n(t) := \frac{A_i^n(t) - \lambda_i^nt}{\sqrt{n}}, \\
    \hat{G}_i^n(t) &:= \frac{G_i^n(t)}{\sqrt{n}}, \quad \hat{R}^n(t) := \frac{R^n(t) - \lambda_0 nt}{\sqrt{n}},  
\end{aligned}
\label{eq: diffusion scales}
\end{equation}
for all $t\geq0$ and $i\in\{1, \cdots, K\}$. 
This, together with \eqref{eq: basic queue length with decomposed arrivals}, implies the alternative process-level diffusion-scaled queue length process 
\begin{equation}
    \hat{Q}_i^n(t) = \hat{Q}_i^n(0) + \hat{A}_i^n(t) + \frac{\lambda_i^n - \lambda_0 n}{\sqrt{n}} t - \hat{L}_i^n(t) -  \hat{G}_i^n(t) - \hat{R}^n(t), 
    \label{eq: diffusion scaled queue length}
\end{equation}
where 
\begin{equation*}
    \hat{R}^n(t) = \min_{1\leq j\leq K} \{\hat{Q}_j^n(0) + \hat{A}_j^n(t) + \frac{\lambda_j^n - \lambda_0 n}{\sqrt{n}} t - \hat{L}_i^n(t) -  \hat{G}_i^n(t) \}. 
\end{equation*}

Let $(\hat{Q}_1^n, \cdots, \hat{Q}_K^n) \geq 0$ denote the diffusion-scaled number of components in the queue at time $t\geq 0$. 
We define the state space as 
\begin{equation}
    E^n := \{(s_1, s_2, \cdots, s_K)\in E_{(1)}^n\times E_{(2)}^n\cdots\times E_{(K)}^n: \prod_{j=1}^K s_j = 0 \}, 
    \label{eq: state space}
\end{equation}
where $E_{(j)}^n := \{0, 1/\sqrt{n}, \cdots, b_j^n/\sqrt{n}\}$ for $j = 1, \cdots, K$, and $s_j$'s denote the scaled queue length of the $j$th queue. 

Under the assumptions above, $(\hat{Q}_1^n, \cdots, \hat{Q}_K^n)$ is a Markov chain on $\mathbb{R}_+^K$ with rate matrix $M$ in twofold. 
First, if $s_j < b_j^n/\sqrt{n}$ for all $j\in\{1, \cdots, K\}$, we have 
\begin{equation}\label{eq:rate matrix 1}
\begin{aligned}
    &\quad M((s_1, \cdots, s_K), (s_1', \cdots, s_K')) \\
    &=
    \begin{cases}
        \lambda_i^n & \text{ if } s_1' = s_1 , \cdots, s_i' = s_i+ \frac{1}{\sqrt{n}}, \cdots, s_K' = s_K, \prod_{j\neq i} s_j = 0, \\
        \lambda_i^n & \text{ if } s_1' = s_1 - \frac{1}{\sqrt{n}}, \cdots, s_i' = s_i, \cdots, s_K' = s_K - \frac{1}{\sqrt{n}}, \prod_{j\neq i} s_j \neq 0, \\
        \delta_i^n\sqrt{n} s_i  & \text{ if } s_1' = s_1, \cdots, s_i' = s_i - \frac{1}{\sqrt{n}}, \cdots, s_K' = s_K, \\
        0 &\text{ otherwise}, 
    \end{cases}
\end{aligned}
\end{equation}
where $(s_1, \cdots, s_K), (s_1', \cdots, s_K')\in E^n$ and $i\in\{1, \cdots, K\}$. 
Second, if $s_i = b_i^n/\sqrt{n}$ for some $i\in\{1, \cdots, K\}$ and $\prod_{j\neq i} s_j = 0$, we have
\begin{equation}\label{eq:rate matrix 2}
\begin{aligned}
    &\quad M((s_1, \cdots, s_K), (s_1', \cdots, s_K')) \\
    &=
    \begin{cases}
        \lambda_j^n & \text{ if } s_1' = s_1 , \cdots, s_j' = s_j + \frac{1}{\sqrt{n}}, \cdots, s_K' = s_K, \prod_{k\neq j,i} s_k = 0, 
        s_j < \frac{b_j^n}{\sqrt{n}}\\
        \lambda_j^n & \text{ if } s_1' = s_1 , \cdots, s_j' = s_j, \cdots, s_K' = s_K, \prod_{k\neq j,i} s_k = 0, 
        s_j = \frac{b_j^n}{\sqrt{n}}\\
        \lambda_j^n & \text{ if } s_1' = s_1 - \frac{1}{\sqrt{n}}, \cdots, s_i' = s_i, \cdots, s_K' = s_K - \frac{1}{\sqrt{n}}, \prod_{k\neq j,i} s_k \neq 0,\\
        \delta_j^n\sqrt{n} s_j & \text{ if } s_1' = s_1, \cdots, s_j' = s_j - \frac{1}{\sqrt{n}}, \cdots, s_K' = s_K, \\
        0 &\text{ otherwise}, 
    \end{cases}
\end{aligned}
\end{equation}
for $j\neq i$. 
Observe that the states $s_i$'s represent the scaled states corresponding to the state space $E^n$. For notational simplicity, we use $s_i$ in place of $\hat{s}_i$ to denote the scaled quantities, with a slight abuse of notation. 
Therefore, a factor of $\sqrt{n}$ appears in the abandonment terms. 

The structure of the rate matrices is not as intricate as it may initially appear, as they primarily reflect the status of each queue. For instance, in \eqref{eq:rate matrix 1}, when a new arrival occurs at the $i$th queue, the state of the other queues must also be taken into account. Specifically, if the $i$th queue is the only empty queue, the new arrival immediately triggers a match. In contrast, if one or more of the other queues are empty and the $i$th queue is not, then the arriving item must wait in its respective queue until a suitable match becomes available in the future. 
In \eqref{eq:rate matrix 2}, if a queue $i$ reaches its buffer limit, any incoming component of category $i$ is rejected due to full capacity. For other categories, an arriving component may either be matched instantaneously or rejected, depending on the state of the corresponding queues.

\section{Asymptotic analysis}
\label{sec: Asymptotic analysis}

This section is devoted to asymptotic analysis by characterizing the weak convergence of diffusion-scaled queue length in $D^k[0, T]$ as $n\to\infty$. 
While one might consider a process-level approach similar to that in \cite{xie2024multi}, the joint dependence among $\hat{L}_i^n$ and $\hat{R}^n$, and the scaled queue length $\hat{Q}_i^n$ prevents a direct extension. 
Additionally, the integral representation admits a non-expansive map with Lipschitz constant one, which further complicates this route. 
Instead, under the Markovian assumptions, we adopt the infinitesimal generator method to establish weak convergence (cf. \cite{kumar2007integrating}, \cite{pang2007martingale}, \cite{budhiraja2011multiscale}, \cite{xie2024diffusion}). We begin by presenting the main result in Theorem \ref{thm: weak convergence}.

\begin{theorem}\label{thm: weak convergence}
    Assume that Assumptions 1-2 hold. Consider the diffusion-scaled queue length process $(\hat{Q}^n)$ with rate matrix \eqref{eq:rate matrix 1} and \eqref{eq:rate matrix 2}. If $\hat{Q}^n(0) \Rightarrow X(0)$, then $\hat{Q}^n(\cdot) \Rightarrow X(\cdot)$ as $n\to\infty$, where the diffusion limit $X$ satisfies
    \begin{equation}\label{eq: limiting process}
        X(t) = X(0) + \int_0^t (\beta - \delta X(s)) \,ds + \int_0^t \Sigma \,dW_s  - U_t, 
    \end{equation}
    for all $t\geq 0$, where $\beta = (\beta_1, \cdots, \beta_K)^\intercal$, $\delta = (\delta_1, \cdots, \delta_K)^\intercal$, $W_t  = (W_1(t), \cdots, W_K(t))^\intercal$ are $K$ independent standard Brownian motions, $\Sigma\in\mathbb{R}^{K\times K}$ is obtained by a decomposition $\Sigma\Sigma^\intercal = a$ and $a\in\mathbb{R}^{K\times K}$ is a diffusion coefficient matrix of the infinitesimal generator of \eqref{eq: limiting process} with elements
    \begin{align*}
        a_{m, m} &= \lambda_0 (\mathbbm{1}_{[\prod_{j\neq m} s_j = 0]} + K - 1 - \sum_{k\neq m} \mathbbm{1}_{[\prod_{j\neq k} s_j = 0]}), \\
        a_{m, n} &= 2\lambda_0 (K - 2 - \sum_{k\neq m, n} \mathbbm{1}_{[\prod_{j\neq k} s_j = 0]}), 
    \end{align*}
    for $m, n\in \{1, \cdots, K\}$. 
    Additionally, $U_t = (U_1(t), \cdots, U_K(t))^\intercal$ and $U_i$'s are the local time processes at buffers $b_i$, respectively. 
\end{theorem}

\begin{proof}
    The proof is divided into three parts. First, we establish an infinitesimal generator of the $n$th process $(\hat{Q}^n)$ and the limiting process. 
    Second, we employ Theorem 6.1 of Chapter 1 in \cite{ethier2009markov} to show the corresponding Feller semigroup on its proper Banach space converges under the notion of convergence \eqref{eq: notion of convergence}.
    Last but not least, we develop the weak convergence result. 
    
    \textit{Step 1}. 
    We establish the infinitesimal generator of our matching queue proposed in \eqref{eq: diffusion scaled queue length} with rate matrices \eqref{eq:rate matrix 1} and \eqref{eq:rate matrix 2}. We provided a clear picture of the construction of a matching system with perishable components in \cite{xie2024multi}. 
    It suffices to integrate the buffer capacities; namely, for each new arrival, we need to check the queue capacity to decide if it can be admitted or blocked. 

    Let the set $E^{n}$ denote all states of $(\hat{Q}^n)$, and it is given by \eqref{eq: state space}. We introduce the Banach space $B(E^n)$ of all bounded measurable functions on $E^n$ with $\|f\|_{B(E^n)} = \sup_{x\in E^n} |f(x)|$. 
    Let $B(\mathbb{R})$ represent the Banach space of all bounded Borel measurable functions on $\mathbb{R}$ equipped with $\|f\|_{B(\mathbb{R})} = \sup_{x\in\mathbb{R}}|f(x)|$. 
    Moreover, let $\hat{C}(\mathbb{R})\subseteq B(\mathbb{R})$ denote the space of all continuous functions that have finite limits at infinity and let $\hat{C}_c^2(\mathbb{R})$ denote the space of all $f\in\hat{C}(\mathbb{R})$ with compact support that is twice continuously differentiable. 
    Since the process $(\hat{Q}^n)$ is a $K$-dimensional Markov chain on $E^n$ with rate matrices \eqref{eq:rate matrix 1} and \eqref{eq:rate matrix 2}, the generator of the $n$th system is given by 
    \begin{equation}\label{eq: generator 1}
        \begin{aligned}
            &A_n f_n(s_1, \cdots, s_K) \\
            &= 
            \sum_{i=1}^K \lambda_i^n \bigg[\left(f_n\left(s_1, \cdots, s_i + \frac{1}{\sqrt{n}}, \cdots, s_K\right) - f_n(s_1, \cdots, s_K)\right) \mathbbm{1}_{[\prod_{j\neq i} s_j = 0]}\\
            &\quad\quad + \left(f_n\left(s_1 - \frac{1}{\sqrt{n}}, \cdots, s_i, \cdots, s_K- \frac{1}{\sqrt{n}}\right) - f_n(s_1, \cdots, s_K)\right) \mathbbm{1}_{[\prod_{j\neq i} s_j \neq 0]}
            \bigg]\\
            & \quad+ \sum_{i=1}^K \delta_i^n\sqrt{n} s_i \left[f_n\left(s_1, \cdots, s_i - \frac{1}{\sqrt{n}}, \cdots, s_K\right) - f_n(s_1, \cdots, s_K)
            \right],
        \end{aligned}
    \end{equation}
    if $s_i < b_i^n/\sqrt{n}$ for all $i$ and $\prod_{k=1}^K s_k = 0$. 
    If $s_i = b_i^n/\sqrt{n}$ for some $i$ and $\prod_{k\neq i} s_k = 0$, we have
    \begin{equation}\label{eq: generator 2}
        \begin{aligned}
            &A_n f_n\left(s_1, \cdots, s_i = \frac{b_i^n}{\sqrt{n}}, \cdots, s_K\right) \\
            &= 
            \sum_{j\neq i} \lambda_j^n \bigg[\left(f_n\left(s_1, \cdots, s_j + \frac{1}{\sqrt{n}}, \cdots, s_K\right) - f_n(s_1, \cdots, s_K)\right) \mathbbm{1}_{[\prod_{k\neq i,j} s_k = 0]}
            \mathbbm{1}_{[s_j < b_j^n/\sqrt{n}]}\\
            &\quad\quad + \left(f_n\left(s_1 - \frac{1}{\sqrt{n}}, \cdots, s_j, \cdots, s_K- \frac{1}{\sqrt{n}}\right) - f_n(s_1, \cdots, s_K)\right) \mathbbm{1}_{[\prod_{k\neq i,j} s_k \neq 0]}
            \bigg]\\
            & \quad+ \sum_{j=1}^K \delta_j^n\sqrt{n} s_j \left[f_n\left(s_1, \cdots, s_j - \frac{1}{\sqrt{n}}, \cdots, s_K\right) - f_n(s_1, \cdots, s_K)
            \right]. 
        \end{aligned}
    \end{equation}
    The domain of the generator $A_n$, $D(A_n)$ is $B(E^n)$, which contains all the bounded measurable functions on $E^n$. 
    By Theorem 1.6 of Chapter 8 in \cite{ethier2009markov}, one can easily justify the positive-definite condition, which guarantees that this generator yields a Feller semigroup $T_n(t)$ on $\hat{C}(\mathbb{R}^K)$ (also, see Remark \ref{rm: matrix a}). 
    Since the state process is a regulated stochastic process, some boundary conditions need to be determined to characterize the reflection at the buffers. We will exhibit this in later steps. 

    Following the same line of thought, the generator of the diffusion limit process \eqref{eq: limiting process} is given by
    \begin{equation}\label{eq: generator limiting process}
        \begin{aligned}
            A f\left(s_1, \cdots, s_K\right) 
            &= \frac{\lambda_0}{2} \sum_{i=1}^K \mathbbm{1}_{[\prod_{j\neq i} s_j = 0]} \frac{\partial^2 f}{\partial s_i^2}
            -\sum_{i=1}^K \beta_i \sum_{j\neq i} \frac{\partial f}{\partial s_j} 
            -\sum_{i=1}^K \delta_i s_i \frac{\partial f}{\partial s_i}\\
            &\quad+ \frac{\lambda_0}{2} \sum_{i=1}^K \sum_{j\neq i} \sum_{k\neq i} \frac{\partial^2 f}{\partial s_j \partial s_k} (1 - \mathbbm{1}_{[\prod_{j\neq i} s_j = 0]}) \\
            &= \frac{1}{2} \sum_{m, n = 1}^K a_{m, n} \frac{\partial^2f}{\partial s_m \partial s_n} + \sum_{m=1}^K (\beta_m - \delta_m s_m)  \frac{\partial f}{\partial s_m}. 
        \end{aligned}
    \end{equation}
    where $f\in D(A):=\{f\in \hat{C}^2(\mathbb{R}^K): f'\text{ is bounded}\}$. 
    We further assume the regulated conditions 
    \begin{equation}\label{eq: regulated conditions 1}
        \sum_{i=1}^K \frac{\partial f}{\partial s_i}(s_1, \cdots, s_K) = 0, 
    \end{equation}
    for all $(s_1, \cdots, s_K)\in E^n$, and at the boundaries
    \begin{equation}\label{eq: regulated conditions 2}
        \frac{\partial f}{\partial s_i}(s_1, \cdots, s_K) = 0
    \end{equation}
    if $s_j = b_j^n/\sqrt{n}$ for some $j\in \{1, \cdots, K\}$. 
    Here, the differential operator $A$ is the infinitesimal generator of a stochastic differential equation \eqref{eq: limiting process}. 
    Observe that the last equality is a more compact form of the generator utilizing the scalar matrices, and the second term $\sum_{j\neq i} \frac{\partial f}{\partial s_j} = -\frac{\partial f}{\partial s_i}$ is obtained by the regulated condition \eqref{eq: regulated conditions 1}.

     \textit{Step 2}. 
     We intend to apply Theorem 6.1 of Chapter 1 in \cite{ethier2009markov}, where it suffices to show that for each $f\in \hat{C}_c^2(\mathbb{R}^K)$, there exists $f_n\in D(A_n)$ such that $f_n\to f$ and $A_nf_n \to Af$ under the notion of convergence \eqref{eq: notion of convergence}. 
     To this end, we pick $f_n = f + \phi/\sqrt{n}$, where $\phi:\mathbb{R}^K \to \mathbb{R}$ and $\phi \in C(\mathbb{R}^K)$. Let $\pi_n$ be the identity transformation. It is natural to show that $\|f_n - \pi_n f\|_{B(E^n)}\to 0$ as $n$ goes to infinity since $f\in \hat{C}_c^2(\mathbb{R}^K)$. 

     Now, we are left to show $A_n f_n \to Af$ as $n\to\infty$. 
     First, we address the case that none of the states reach the buffers, namely $s_i < b_i^n/\sqrt{n}$ for all $i\in\{1, \cdots, K\}$. 
     Since the definition of $f_n$ and $f\in \hat{C}_c^2(\mathbb{R}^K)$, the Taylor's expansion suggests
     \begin{equation}\label{eq: Taylor}
     \begin{aligned}
        &\quad f\left(s_1, \cdots, s_i + \frac{1}{\sqrt{n}}, \cdots, s_K\right) \\
        &= f(s_1, \cdots, s_K) + \frac{1}{\sqrt{n}}\frac{\partial f}{\partial s_i} + \frac{1}{2} \left(\frac{1}{\sqrt{n}}\right)^2 \frac{\partial^2 f}{\partial s_i^2} + o\left(\frac{1}{n}\right), \\
        &\quad f\left(s_1 - \frac{1}{\sqrt{n}}, \cdots, s_i, \cdots, s_K - \frac{1}{\sqrt{n}}\right)\\
        &= f(s_1, \cdots, s_K) - \frac{1}{\sqrt{n}}\sum_{j\neq i} \frac{\partial f}{\partial s_j} 
        + \frac{1}{2!} \frac{1}{n} \sum_{j\neq i}\sum_{k\neq i} \frac{\partial^2 f}{\partial s_j \partial s_k} + o\left(\frac{1}{n}\right). 
    \end{aligned}
    \end{equation}
    Using proper Taylor's expansions \eqref{eq: Taylor}, a simple algebraic manipulation suggests that for $s_i < b_i^n/\sqrt{n}$, 
    \begin{equation}\label{eq: generator 1 with Taylor}
        \begin{aligned}
            &\quad A_nf_n(s_1, \cdots, s_K) \\
            &= \frac{1}{2} \sum_{i=1}^K \sum_{j\neq i} \sum_{k\neq i} \frac{\partial^2 f}{\partial s_j \partial s_k} \frac{\lambda_i^n}{n} 
            + \frac{1}{2}\sum_{i=1}^K \frac{\delta_i^n s_i}{\sqrt{n}} \frac{\partial^2 f}{\partial s_i^2} \\
            &\quad + \sum_{i=1}^K \mathbbm{1}_{[\prod_{j\neq i} s_j = 0]} \left(\frac{1}{2} \frac{\partial^2 f}{\partial s_i^2} \frac{\lambda_i^n}{n} - \frac{1}{2} \sum_{j\neq i}\sum_{k\neq i} \frac{\partial^2 f}{\partial s_j \partial s_k} \frac{\lambda_i^n}{n}\right) \\
            &\quad + \sum_{i=1}^K \sum_{j\neq i} \frac{\partial f}{\partial s_j} \left( - \frac{\lambda_i^n - \lambda_0 n}{\sqrt{n}}\right)
            + \sum_{i=1}^K \mathbbm{1}_{[\prod_{j\neq i} s_j = 0]}\frac{\partial f}{\partial s_i} \frac{\lambda_i^n - \lambda_0 n}{\sqrt{n}}\\
            &\quad + \sum_{i=1}^K \mathbbm{1}_{[\prod_{j\neq i} s_j = 0]} \sum_{j\neq i} \frac{\partial f}{\partial s_j} \frac{\lambda_i^n - \lambda_0 n}{\sqrt{n}} 
            -\sum_{i=1}^K \delta_i^n s_i \frac{\partial f}{\partial s_i}
            + \frac{1}{\sqrt{n}}A_n \phi, 
        \end{aligned}
    \end{equation}
    where $A_n\phi$ is defined as the same way as \eqref{eq: generator 1 with Taylor} with $f$ substituted by $\phi$ and for simplicity, we omit its details. 
    Note that the $\delta_i^n/\sqrt{n}$ term is negligible since $\lim_{n\to\infty} \delta_i^n = \delta_i < \infty$. 
    
    Using \eqref{eq: generator limiting process} and \eqref{eq: generator 1 with Taylor}, we further obtain that for $f\in \hat{C}_c^2(\mathbb{R}^K)$ and $s_i < b_i^n / \sqrt{n}$ for all $i\in\{1, \cdots, K\}$, 
    \begin{align*}
        &\quad \left|A_nf_n(s_1, \cdots, s_K) - (\pi_nAf)(s_1, \cdots, s_K)\right| \\
        & \leq 
        \sum_{i=1}^K \sum_{j\neq i} \sum_{k\neq i} \left| \frac{1}{2} \frac{\partial^2 f}{\partial s_j \partial s_k} (1 - \mathbbm{1}_{[\prod_{j\neq i} s_j = 0]}) \left( \frac{\lambda_i^n}{n} - \lambda_0 \right) \right| \\
        &\quad + \frac{1}{2} \sum_{i=1}^K \mathbbm{1}_{[\prod_{j\neq i} s_j = 0]} \left| \frac{\partial^2 f}{\partial s_i^2} \left(\frac{\lambda_i^n}{n} - \lambda_0\right) \right| \\
        & \quad + \sum_{i=1}^K \left|\sum_{j\neq i} \frac{\partial f}{\partial s_j} \left( - \frac{\lambda_i^n - \lambda_0 n}{\sqrt{n}} + \beta_i\right)\right|
        + \sum_{i=1}^K |\delta_i^n - \delta_i| s_i \frac{\partial f}{\partial s_i}
        \\
        & \quad + \sum_{i=1}^K \mathbbm{1}_{[\prod_{j\neq i} s_j = 0]} 
        \left| \sum_{j\neq i} \frac{\partial f}{\partial s_j} + \frac{\partial f}{\partial s_i}\right|
        \left|\frac{\lambda_i^n - \lambda_0 n}{\sqrt{n}} \right|
        + \frac{1}{\sqrt{n}} |A_n \phi|. 
    \end{align*}
    Observe that the summation $\sum_{j\neq i} \frac{\partial f}{\partial s_j} + \frac{\partial f}{\partial s_i} = \sum_{j = 1}^K \frac{\partial f}{\partial s_j}$, which turns out to be zero due to the regulated condition \eqref{eq: regulated conditions 1}. 
    Hence, taking supremum over $E^n$ and as $n\to\infty$, we have for each $f\in \hat{C}_c^2(\mathbb{R}^K)$ and $s_i < b_i^n / \sqrt{n}$, 
    \[
     \|A_n f_n - \pi_n Af\|_{B(E^n)} = \sup_{(s_1, \cdots, s_K)\in E^n} \left|(A_nf_n - \pi_nAf)(s_1, \cdots, s_K)\right| \to 0. 
    \]

    Second, we address the case that at least one queue reaches their buffers, namely, if $s_i = b_i^n/\sqrt{n}$ for some $i$'s and $\prod_{k\neq i}s_k = 0$, namely, the $i$th queue reaches the buffer and at least one empty queue for the rest of the queues, Taylor's expansions yield 
    \begin{equation}\label{eq: generator 2 with Taylor}
        \begin{aligned}
            &\quad A_nf_n\left(s_1, \cdots, s_i = \frac{b_i^n}{\sqrt{n}}, \cdots, s_K\right) \\
            &= \frac{1}{2} \sum_{j\neq i} \mathbbm{1}_{[\prod_{k\neq i, j} s_k = 0]} 
            \mathbbm{1}_{[s_j < b_j^n/\sqrt{n}]}
            \frac{\partial^2 f}{\partial s_j^2} \frac{\lambda_j^n}{n} \\
            &+ \frac{1}{2}\sum_{j\neq i} \frac{\lambda_j^n}{n} \sum_{l\neq j}\sum_{k\neq j} \frac{\partial^2 f}{\partial s_l \partial s_k} 
            - \frac{1}{2} \sum_{j\neq i} \mathbbm{1}_{[\prod_{k\neq i, j} s_k = 0]} \frac{\lambda_j^n}{n} \sum_{l\neq j}\sum_{k\neq j} \frac{\partial^2 f}{\partial s_l \partial s_k}\\
            &+ \frac{1}{2}\sum_j \frac{\delta_j^n s_j}{\sqrt{n}} \frac{\partial^2 f}{\partial s_j^2}
            - \sum_j \delta_j^n s_j \frac{\partial f}{\partial s_j}
            + \sum_{j\neq i} 
            \mathbbm{1}_{[\prod_{k\neq i, j} s_k = 0]} 
            \mathbbm{1}_{[s_j < b_j^n/\sqrt{n}]}
            \frac{\partial f}{\partial s_j} \frac{\lambda_j^n - \lambda_0 n}{\sqrt{n}} \\
            &- \sum_{j\neq i} \frac{\lambda_j^n - \lambda_0 n}{\sqrt{n}} \sum_{l\neq j} \frac{\partial f}{\partial s_l} 
            + \sum_{j\neq i} \mathbbm{1}_{[\prod_{k\neq i, j} s_k = 0]} \frac{\lambda_j^n - \lambda_0 n}{\sqrt{n}} \sum_{l\neq j} \frac{\partial f}{\partial s_l}
            + \frac{1}{\sqrt{n}}A_n \phi, 
        \end{aligned}
    \end{equation}
    where $A_n \phi$ is defined the same way as \eqref{eq: generator 2 with Taylor} with $f$ substituted by $\phi$. Additionally, we adopt the boundary conditions $\frac{\partial f}{\partial s_j}(s_1, \cdots, s_i = b_i^n/\sqrt{n}, \cdots, s_K) = 0$ for all $j$ and some $i$ as in \eqref{eq: regulated conditions 2}, namely, partial derivatives vanish if at least one state reaches the capacity. 

    Now, \eqref{eq: generator limiting process} and \eqref{eq: generator 2 with Taylor} suggest that when $s_i = b_i^n / \sqrt{n}$ for some $i$'s, 
    \begin{align*}
        &\quad \left|A_nf_n\left(s_1, \cdots, s_i = \frac{b_i^n}{\sqrt{n}}, \cdots, s_K\right) - (\pi_nAf)\left(s_1, \cdots, s_i = \frac{b_i^n}{\sqrt{n}}, \cdots, s_K\right)\right| \\
        & \leq \frac{1}{2} \sum_{j\neq i}  \mathbbm{1}_{[\prod_{k\neq i, j} s_k = 0]}  \mathbbm{1}_{[s_j < b_j^n/\sqrt{n}]}\left|\frac{\partial^2 f}{\partial s_j^2} \left(\frac{\lambda_j^n}{n} - \lambda_0\right)\right| \\
        &\quad +\frac{1}{2} \sum_{j\neq i} \sum_{l\neq j}\sum_{k\neq j} \left|\frac{\lambda_j^n}{n} - \lambda_0\right| 
        \left|\frac{\partial^2 f}{\partial s_l \partial s_k} \left(1 - \mathbbm{1}_{[\prod_{k\neq i, j} s_k = 0]}\right)\right| \\
        &\quad + \bigg| 
        \frac{\lambda_0}{2} \sum_{k\neq i}\sum_{l\neq i}  \frac{\partial^2 f}{\partial s_k \partial s_l} \left( 1 - \mathbbm{1}_{[\prod_{k\neq i} s_k = 0]}\right) 
        + \frac{\lambda_0}{2} \sum_{j=1}^K \mathbbm{1}_{[\prod_{k\neq i, j} s_k = 0]} \mathbbm{1}_{[s_j < b_j^n/\sqrt{n}]} \frac{\partial^2f}{\partial s_j^2} \\
        &\quad\quad - \frac{1}{2\sqrt{n}} \sum_{j\neq i} \mathbbm{1}_{[\prod_{k\neq i, j} s_k = 0]} \mathbbm{1}_{[s_j<b_j^m/\sqrt{n}]} \frac{\partial^2 \phi}{\partial s_j^2} \frac{\lambda_j^n}{n} \\
        &\quad \quad - \frac{1}{2\sqrt{n}} \sum_{j\neq i} \frac{\lambda_j^n}{n} \sum_{l\neq j}\sum_{k\neq j} \frac{\partial^2 \phi}{\partial s_l \partial s_k} \left(1 - \mathbbm{1}_{[\prod_{k\neq i, j} s_k = 0]}\right)
        \bigg|. 
    \end{align*}
    Here we selection $\phi\in C(\mathbb{R}^K)$ such that the last absolute value term vanishes to compensate the continuity of the generator at the boundaries (see \cite{xie2022topics} and \cite{xie2024multi}), which we called the compensation condition. Additionally, $\phi$ satisfies the regulated conditions \eqref{eq: regulated conditions 1} and \eqref{eq: regulated conditions 2}. 
    Therefore, taking supremum over $E^n$ and as $n\to\infty$, we obtain the convergence of the generator at the boundaries under the notion of convergence \eqref{eq: notion of convergence} for each $f\in \hat{C}_c^2(\mathbb{R}^K)$ and $s_i = b_i^n / \sqrt{n}$ for some $i$'s. 
    Since the supremum is taken over $E^n \cap \text{supp}(f)$, the convergence $A_n f_n \to Af$ as $n\to\infty$ follows under the notion of convergence \eqref{eq: notion of convergence}.

     \textit{Step 3}. 
     We have shown that for each $f\in \hat{C}_c^2(\mathbb{R}^K)$ there exists $f_n = f+ \phi/\sqrt{n}$, where $\phi$ satisfies the compensation condition and regulated conditions, such that $f_n \to f$ and $A_n f_n \to Af$ under the notion of convergence \eqref{eq: notion of convergence}. Therefore, Theorem 6.1 of Chapter 1 in \cite{ethier2009markov} suggests that for each $f\in B(\mathbb{R}^K)$, $T_n(t) \pi_nf \to T(t) f$ under the notion of convergence for all $t\geq 0$, where $\{T_n(t)\}$ is a Feller semigroup on $\hat{C}(\mathbb{R}^K)$ generated by \eqref{eq: generator 1} and $\{T(t)\}$ is a Feller semigroup on $\hat{C}(\mathbb{R}^K)$ generated by \eqref{eq: generator limiting process}. Moreover, the map $\pi_n$ is the identify map. 
     Therefore, Theorem 2.11 of Chapter 4 in \cite{ethier2009markov} further yields that there exists a Markov process $X$ corresponding to $\{T(t)\}$ with initial distribution $X(0)\in \mathbb{R}^K$ such that $\hat{Q}^n \Rightarrow X$ as desired. 

\end{proof}

\begin{remark}\label{rm: matrix a}
    The Feller semigroup of the generator $A$ in Step 3 is ensured by Theorem 1.6 and Theorem 2.5 in Chapter 8 of \cite{ethier2009markov}, where it suffices to either justify that the coefficient matrix $a$ satisfies $\inf_{|\theta| = 1} \theta\cdot a\theta > 0$ or the boundedness of partial derivatives of the coefficients. 
    To simplify the exposition and without loss of generality, we may consider the system of the case $K=3$ without abandonment and buffers, where its generator \eqref{eq: generator 1} and \eqref{eq: generator 2} can be significantly simplified. Moreover, this instance provides a clear picture of the structures of the generators and the connections of the scaled state process and its limiting process. 
    In this case, the matrix $a(s_1, s_2, s_3)=[a_{ij}]$ has elements: 
    \begin{align*}
        a_{11} &= \lambda_0 (2 + \mathbbm{1}_{[s_2s_3 = 0]} - \mathbbm{1}_{[s_1s_3 = 0]} - \mathbbm{1}_{[s_1s_3=0]}), \\
        a_{22} &=  \lambda_0(2 + \mathbbm{1}_{[s_1s_3= 0]} - \mathbbm{1}_{[s_2s_3 = 0]} - \mathbbm{1}_{[s_1s_2 = 0]}), \\
        a_{33} & = \lambda_0(2 + \mathbbm{1}_{[s_1s_2 = 0]} - \mathbbm{1}_{[s_2s_3 = 0]} - \mathbbm{1}_{[s_1s_3 = 0]}), \\
        a_{12} &\equiv a_{21} = 2\lambda_0 (1 - \mathbbm{1}_{[s_1s_2 = 0]}), \\
        a_{13} &\equiv a_{31} = 2\lambda_0 (1 - \mathbbm{1}_{[s_1s_3 = 0]}), \\
        a_{23} &\equiv a_{32} = 2\lambda_0 (1 - \mathbbm{1}_{[s_2s_3 = 0]}).
    \end{align*}
    
    For nondegenerate diffusion cases, the matrix $a$ becomes an identity matrix, which is a positive-definite matrix such that $\inf_{|\theta| = 1} \theta\cdot a\theta > 0$ holds. This occurs when at least two queues are empty, for instance, $s_1 = s_2 = 0$ and $s_3\neq 0$. 
    For degenerate diffusion cases, we may reduce the dimensionality of the matrix $a$. For example, consider the cases of exactly one empty queue and $\{(s_1, s_2, s_3) \ | \ s_1 = 0,\ s_2\neq 0, \ s_3 \neq 0\} \subseteq E^n$. The generator \eqref{eq: generator limiting process} becomes 
    \[
    Af(s) = \frac{1}{2} \sum_{i, j =1}^3 a_{ij} \partial_i \partial_j f + \sum_{i=1}^3 b_i \partial_i f,
    \]
    where the diffusion and drift matrices are 
    \[
    a(s) = \begin{bmatrix}
        0 & 0 & 0 \\
        0 & 2\lambda_0 & 2\lambda_0 \\
        0 & 2\lambda_0 & 2\lambda_0
    \end{bmatrix}, \
    b(s) = \begin{bmatrix}
        b_1(s)\\
        b_2(s)\\
        b_3(s)
    \end{bmatrix}
    = \begin{bmatrix}
        \beta_1 - \delta_1 s_1 \\
        \beta_2 - \delta_2 s_2 \\
        \beta_3 - \delta_3 s_3
    \end{bmatrix}. 
    \]
    By the boundary condition \eqref{eq: regulated conditions 1}, we observe that $\partial_1 f = - \partial_2 f - \partial_3 f$, which further deduce 
    \[
    Af(s) = \frac{1}{2} \sum_{i, j = 1}^2 \Tilde{a}_{ij} \partial_i \partial_j f + \sum_{i=1}^2 \Tilde{b}_i \partial_i f, 
    \]
    where $\Tilde{a}(s)$ is a submatrix of $a(s)$ with nonzero elements and $\Tilde{b}(s) = [b_2 - b_1, b_3 - b_1]^\intercal$. It is straightforward to verify the conditions in Theorem 8.2.5 in \cite{ethier2009markov}. 
\end{remark}

To better understand the limiting process \eqref{eq: limiting process}, we examine the case $K=2$, which corresponds to a double-ended queue.

\begin{example}\label{ex1}
    In Theorem \ref{thm: weak convergence} and when the number of categories is $K = 2$, the heavy-traffic limiting process \eqref{eq: limiting process} admits the following matrix form: 
    \begin{equation}\label{eq: ex1 limiting process}
        \begin{bmatrix}
            X_1(t) \\
            X_2(t)
        \end{bmatrix} 
        = \begin{bmatrix}
            x_1\\
            x_2
        \end{bmatrix}
         + \Sigma W_t + \begin{bmatrix}
             \int_0^t (\beta_1 - \delta_1 X_1(s))\, ds\\
             \int_0^t (\beta_2 - \delta_2 x_2(s))\, ds
         \end{bmatrix}
          - \begin{bmatrix}
              U_1(t)\\
              U_2(t)
          \end{bmatrix}. 
    \end{equation}
    If we consider the difference of two elements, it is straightforward to deduce a double-ended queue for $t\geq 0$,
    \begin{equation}\label{eq: ex1 double-ended queue}
        X(t) = x + \sigma B_t + \int_0^t (\beta - h(X(s)))\, ds - U_t, 
    \end{equation}
    where $X(t) = X_1(t) - X_2(t)$, $\sigma B_t$ and $[\Sigma W_t]_1 - [\Sigma W_t]_2$ are identically distributed, $\beta = \beta_1 - \beta_2$, $h(x) = \delta_1 x_1 - \delta_2 x_2$, and $U_t = U_1(t) - U_2(t)$. 
    Note that since $x = x_1 - x_2$ and the fact that at least one queue is empty at any given time instant, we may express $x_1 = x^+$ and $x_2 = x^-$, which further deduce $h(x) = \delta_1 x^+ - \delta_2 x^-$. 
    Observe that this coincides with the heavy-traffic limit obtained in \cite{liu2021admission}, where their Assumption 2.5 accommodates constant buffers and Definition A.1 ensures the existence of two-sided Skorokhod maps (also, see \cite{kruk2007explicit} and \cite{burdzy2009skorokhod}). 
\end{example}

Some comments on Theorem \ref{thm: weak convergence} are in order. 
First, the limiting diffusion process in Theorem \ref{thm: weak convergence} is more intricate than it may initially appear. Due to the structure of instantaneous matching, we know that at any given time $t>0$, at least one queue must be empty, ensuring that $\prod_{i=1}^K X_i(t) = 0$ for all $t\geq 0$. Otherwise, matches would occur, contradicting the persistence of this situation. 
This further establishes a coupling behavior of the state processes that elements are mutually coupled (cf. \cite{xie2024multi}). 
In the previous Example \ref{ex1}, the coupling behavior is preserved in the limiting equation \eqref{eq: ex1 double-ended queue}. However, the condition that at least one queue remains empty at all times is not directly apparent from the limiting process described in \eqref{eq: ex1 limiting process}. We suspect that a specific decomposition $\Sigma$ of the diffusion matrix $a$ is necessary to characterize and reflect this structural requirement.

Second, the local time process $U_t$ reflects the impact of finite buffer capacities. It activates when a queue hits its buffer limit, introducing reflection. However, a key open question remains: whether the limiting diffusion process described in \eqref{eq: limiting process} is distributionally equivalent to the heavy traffic limit derived in \cite{xie2024multi}, which is expected to be characterized using a regulated coupling stochastic integral equation (cf. \cite{xie2024multi} and \cite{xiewucontrol}):  
\begin{equation}
        Q(t) = Q(0) + \beta t + \sigma W(t) - \int_0^t \delta \diamond Q(s)ds - R(t) I - L(t), 
        \label{eq: heavy traffic limit with finite capacity}
\end{equation}
where $\beta = (\beta_1, \cdots, \beta_K)^\intercal$, $\sigma = \text{diag}(\sigma_1, \cdots, \sigma_K)$, $\delta = (\delta_1, \cdots, \delta_K)^\intercal$, $I = (1, \cdots, 1)^\intercal$, the product $\diamond$ is a Hadamard entrywise product, and the matching completion process
\[
R(t) = \min_{1\leq i \leq K}\{Q_j(0) + \beta_j t + \sigma_j W_j(t) - \int_0^t \delta_j X_j(s)ds -L_j(t)\},
\]
and $L = (L_1, \cdots, L_K)^\intercal$ for each $i\in\{1, \cdots, K\}$ that $\int_0^{t} \mathbbm{1}_{[X_i(s)<b_i]}(s) dL_i(s) = 0$. 
Note that a Skorokhod mapping could be established to account for the regulation in \eqref{eq: heavy traffic limit with finite capacity}. However, it reveals a non-expensive map that cannot ensure a contraction mapping. 

Additionally, if we consider the difference of any two elements in \eqref{eq: heavy traffic limit with finite capacity}, we can deduce that for $i\neq j$
\begin{align*}
    Q_i(t) - Q_j(t) &= Q_i(0) - Q_j(0) + \int_0^t ((\beta_i - \delta_i Q_i(s)) - (\beta_j - \delta_j Q_j(s)))\, ds \\
    &\quad + \sigma_i W_i(t) - \sigma_j W_j(t) - (L_i - L_j)(t), 
\end{align*}
which is identical with the difference of any two elements of \eqref{eq: limiting process}. 
Intuitively, the process $R(t)$ can be interpreted as a drifted Brownian motion combined with a scaled state-dependent term. Notably, the associated local time process remains inactive whenever the corresponding state hits the origin. This observation implies that $R(\cdot)$ admits a semimartingale decomposition, which helps reveal the structural form of the limiting process in \eqref{eq: limiting process}, in relation to the regulated system dynamics in \eqref{eq: heavy traffic limit with finite capacity}. 
One can further distinguish between long queues, which actively contribute to matching, and short queues, which tend to remain empty. In particular, when a queue stays empty, its corresponding local time process is not triggered. As a result, the matching completion process $R(\cdot)$ simplifies to a drifted Brownian motion plus a scaled state-dependent component. We refer to \cite{xie2024multi} for an example of a detailed development of such decomposition. However, a semimartingale decomposition of \eqref{eq: heavy traffic limit with finite capacity} is unknown.

Third, an appropriate decomposition of the diffusion matrix $a = \Sigma \Sigma^\intercal$ is uncertain. As noted in Remark \ref{rm: matrix a}, in certain special cases, the matrix $a$ may fail to be positive definite, in which case a Cholesky decomposition does not exist. 
It is also straightforward to see that the matrix $a$ is positive semi-definite that ensures real-valued non-negative eigenvalues, which further provides an eigenvalue decomposition. 
As mentioned in our initial discussion, it is necessary to address the structural requirement, which may be inferred from a specific matrix decomposition.

Despite these uncertainties, the limiting diffusion should capture two essential features:
(i) When a queue reaches zero, it may remain at zero for a nontrivial duration, depending on the availability of matching components in other queues.
(ii) When a queue reaches its buffer capacity, a reflective behavior must occur.
To formalize these behaviors, it is necessary to establish that the matching completion process defined in \cite{xie2024multi} admits a semimartingale representation consistent with \eqref{eq: limiting process}, involving nontrivial local time terms. 
In \cite{xie2024multi}, a semimartingale decomposition was derived for the special case without abandonment or reflection. However, the equivalence in distribution with the general limiting process remains unresolved. 
In our follow-up work \cite{xiewucontrol}, we will develop and address the novel regulated coupling stochastic integral equation \eqref{eq: heavy traffic limit with finite capacity} in great detail.

 \bibliographystyle{elsarticle-num} 
 \bibliography{cas-refs}





\end{document}